\documentclass[11pt,letterpaper]{article}
\usepackage{amsmath,amssymb,amsthm}
\usepackage{MnSymbol} 
\usepackage{graphicx,color}
\usepackage[hyphens]{url}
\usepackage{dsfont}
\usepackage{booktabs}
\usepackage[square,sort,numbers]{natbib}
\usepackage[sort,nocompress,noadjust]{cite}
\usepackage[normalem]{ulem}
\usepackage{mathtools}
\usepackage[nameinlink,capitalize]{cleveref}
\usepackage{multirow}
\usepackage{algorithm}
\usepackage[noend]{algpseudocode}
\usepackage{xspace}
\usepackage{tikz}

\usepackage{tikz-cd}
\usepackage{caption}
\usepackage{subcaption}

\numberwithin{equation}{section}
\newtheorem{theorem}{Theorem}[section]
\newtheorem{cor}[theorem]{Corollary}
\newtheorem{lemma}[theorem]{Lemma}
\newtheorem{remark}[theorem]{Remark}
\newtheorem{prop}[theorem]{Proposition}

\newtheorem{conj}[theorem]{Conjecture}
\newtheorem{defin}[theorem]{Definition}

\newcommand{\cM}{\mathcal{M}}

\newcommand{\cS}{\mathcal{S}}

\newcommand{\R}{\mathbb{R}}

\newcommand{\Rn}{\R^n}

\newcommand{\Rntk}{(\R^n)^{\otimes k}}
\newcommand{\Rpntk}{(\R_{\geq 0}^n)^{\otimes k}}

\newcommand*{\E}{\mathbb{E}}

\providecommand{\rank}{\operatorname{rank}}

\providecommand{\conv}{\operatorname{conv}}

\newcommand{\bone}{\mathbf{1}}

\newcommand*{\poly}{\mathrm{poly}}

\newcommand*{\eps}{\varepsilon}

\renewcommand{\leq}{\leqslant}
\renewcommand{\geq}{\geqslant}

\newcommand{\MOT}{\textsf{MOT}}

\newcommand{\MinO}{\textsf{MIN}}

\newcommand{\Cmax}{C_{\max}}

\newcommand{\Coup}{\cM(\mu_1,\dots,\mu_k)}
\newcommand{\Coupp}{\cM(\mu_1',\dots,\mu_k')}

\newcommand{\jvec}{\vec{j}}

\DeclareMathOperator{\Ber}{Ber}

\let\baraccent=\= %
\renewcommand{\=}[1]{\stackrel{#1}{=}} %

\providecommand{\RR}{\mathbb{R}}

\providecommand{\ZZ}{\mathbb{Z}}

\providecommand{\cM}{\mathcal{M}}

\providecommand{\eps}{\varepsilon}

\renewcommand{\P}{\mathsf{P}}
\providecommand{\NP}{\mathsf{NP}}

\providecommand{\BPP}{\mathsf{BPP}}

\mathchardef\mhyphen="2D %

\providecommand{\sm}{\setminus}

\newcommand{\interior}[1]{%
	{\kern0pt#1}^{\mathrm{o}}%
}

\usepackage{fancybox}
\newenvironment{fminipage}{\begin{Sbox}\begin{minipage}}{\end{minipage}\end{Sbox}\fbox{\TheSbox}}

\usepackage[margin=1in]{geometry}
	
\makeatletter
\def\blfootnote{\gdef\@thefnmark{}\@footnotetext}
\makeatother

\begin{document}
	\title{Hardness results for Multimarginal Optimal Transport problems}
	\author{Jason M. Altschuler \and Enric Boix-Adser\`a}
	\date{}
	\maketitle

\blfootnote{The authors are with the Laboratory for Information and Decision Systems (LIDS), Massachusetts Institute of Technology, Cambridge MA 02139. Work partially supported by NSF Graduate Research Fellowship 1122374, a Siebel PhD Fellowship, and a TwoSigma PhD fellowship.}

\begin{abstract}
	Multimarginal Optimal Transport ($\MOT$) is the problem of linear programming over joint probability distributions with fixed marginals. A key issue in many applications is the complexity of solving $\MOT$: the linear program has exponential size in the number of marginals $k$ and their support sizes $n$. A recent line of work has shown that $\MOT$ is $\poly(n,k)$-time solvable for certain families of costs that have $\poly(n,k)$-size implicit representations. However, it is unclear what further families of costs this line of algorithmic research can encompass. In order to understand these fundamental limitations, this paper initiates the study of intractability results for $\MOT$.
	\par Our main technical contribution is developing a toolkit for proving $\NP$-hardness and inapproximability results for $\MOT$ problems. We demonstrate this toolkit by using it to establish the intractability of a number of $\MOT$ problems studied in the literature that have resisted previous algorithmic efforts. 
	For instance, we provide evidence that repulsive costs make $\MOT$ intractable by showing that several such problems of interest are $\NP$-hard to solve---even approximately.
\end{abstract}

\section{Introduction}\label{sec:intro}

Multimarginal Optimal Transport ($\MOT$) is the problem of linear programming over joint probability distributions with fixed marginal distributions. That is, given $k$ marginal distributions $\mu_1, \dots, \mu_k$ in the simplex\footnote{For notational simplicity, all $\mu_i$ are assumed to have the same support size. The general case is a straightforward extension.} $\Delta_n = \{u \in \R_{\geq 0}^n : \sum_{i=1}^n u_i = 1 \}$ and a cost tensor $C$ in the $k$-fold tensor product space $\Rntk = \R^n \otimes \cdots \otimes \R^n$, compute
\begin{align}
	\min_{P \in \Coup} \langle P, C \rangle
	\label{MOT}
	\tag{MOT}
\end{align}
where $\Coup$ is the ``transportation polytope'' containing entrywise non-negative tensors $P \in (\R^{n})^{\otimes k}$ 
satisfying the marginal constraints $\sum_{j_1,\dots,j_{i-1}, j_{i+1}, \dots, j_{k}} P_{j_1, \dots, j_{i-1}, j, j_{i+1}, \dots, j_k} = [\mu_i]_j$ for all $i \in \{1, \dots,k\}$ and $j \in \{1, \dots, n\}$.

\par This $\MOT$ problem has recently attracted significant interest due to its many applications in data science, applied mathematics, and the natural sciences; see for instance~\citep{AltBoi20mot,BenCarCut15,Pas15,PeyCut17} and the many references within. However, a key issue that dictates the usefulness of $\MOT$ in applications is its complexity. Indeed, while $\MOT$ can be easily solved in $n^{\Theta(k)}$ time since it is a linear program in $n^k$ variables and $n^k + nk$ constraints, this is far from scalable. 

\par In this paper and the literature, we are interested in ``polynomial time'' algorithms, where polynomial means in the number of marginals $k$ and their support sizes $n$ (as well as the scale-invariant quantity $\|C\|_{\max}/\eps$ if we are considering $\eps$ additive approximations). 
An obvious obstacle is that in general, one cannot even read the input to $\MOT$---let alone solve $\MOT$---in $\poly(n,k)$ time since the cost tensor $C$ has $n^k$ entries.

\par Nevertheless, in nearly all applications of practical interest, the cost tensor $C$ has a simple structure that enables it to be input implicitly via a $\poly(n,k)$-sized representation. Moreover, a recent line of work has shown that in many applications where $C$ has such a polynomial-size implicit representation, the $\MOT$ problem can also be solved in polynomial time.
A simple-to-describe illustrative example is cost tensors $C$ which have constant rank and are given as input in factored form~\citep{AltBoi20mot}. 
Other examples include computing generalized Euler flows~\citep{BenCarCut15,AltBoi20mot}, computing low-dimensional Wasserstein barycenters~\citep{CarObeOud15,AltBoi20bary}, solving $\MOT$ problems with tree-structured costs~\citep{h20tree}, and solving $\MOT$ problems with decomposable costs~\citep{AltBoi20mot}.

\par A fundamental question is: What further families of succinctly representable costs lead to tractable $\MOT$ problems? As an illustrative example, can $\MOT$ still be solved in $\poly(n,k)$ time if the cost $C$ has rank that is low but not constant, say of size $\poly(n,k)$? There are a number of $\MOT$ problems studied in the literature for which $C$ has a $\poly(n,k)$-sized representation but developing $\poly(n,k)$-time algorithms has resisted previous efforts. 
The purpose of this paper is to understand the fundamental limitations of this line of algorithmic research.

\subsection{Contributions}\label{ssec:intro:cont}

This paper initiates the study of intractability results for $\MOT$. Our main contributions are:
\begin{enumerate}
	\item In \S\ref{sec:reduction}, we develop a toolkit for proving $\NP$-hardness and inapproximability results for $\MOT$ problems on costs $C$ that have $\poly(n,k)$-size implicit representations. This is the main technical contribution of the paper. 
	\item In \S\ref{sec:applications:lr}, \S\ref{sec:applications:pairwise}, and \S\ref{sec:applications:repulsive}, we demonstrate this toolkit by using it to establish the intractability of a number of $\MOT$ problems studied in the literature that have resisted previous algorithmic efforts.
\end{enumerate}
We elaborate on each point below.

\subsubsection{Reduction toolkit}

Let $\MOT_C(\mu)$ denote the problem of computing the optimal value of $\MOT$ for a cost tensor $C \in \Rntk$ and marginals $\mu = (\mu_1, \dots, \mu_k) \in (\Delta_n)^k$. Informally, our main result establishes that for any fixed cost $C$, the following discrete optimization problem $\MinO_C$ can be (approximately) solved in polynomial time if $\MOT_C$ can be (approximately) solved in polynomial time. 

\begin{defin}\label{def:min}
	For $C \in \Rntk$ and $p = (p_1, \dots, p_k) \in \R^{n \times k}$, the problem $\MinO_C(p)$ is to compute 
	\begin{align}
		\min_{(j_1, \dots, j_k) \in \{1, \dots, n\}^k} C_{j_1, \dots, j_k} - \sum_{i=1}^k [p_i]_{j_i}. 
		\label{eq:def:min}
	\end{align}
\end{defin}

\par The upshot of our result is that it enables us to prove intractability results for $\MOT_C$ by instead proving intractability results for $\MinO_C$. This is helpful since $\MinO_C$ is more directly amenable to $\NP$-hardness and inapproximability reductions because it is phrased as a more conventional combinatorial optimization problem; examples in \S\ref{sec:applications:lr}, \S\ref{sec:applications:pairwise}, and \S\ref{sec:applications:repulsive}.

\par We briefly highlight the primary insight behind the proof: The convex relaxation of $\MinO_C$ is exact and is a convex optimization problem whose objective can be evaluated by solving an auxiliary $\MOT_C$ problem. 
This means that if one can (approximately) solve $\MOT_C$, then one can use this in a black-box manner to (approximately) solve $\MinO_C$ via zero-th order optimization. 
In \S\ref{sec:reduction}, we show how to perform this zero-th order optimization efficiently using the Ellipsoid algorithm if $\MOT_C$ can be computed exactly, and otherwise using recent developments on zero-th order optimization of approximately convex functions~\citep{Ris16,Bel15} if $\MOT_C$ can be computed approximately. 

\par We conclude this discussion with several remarks.

\begin{remark}[Converse]
	There is no loss of generality when using our results to reduce proving the intractability of $\MOT_C$ to proving the intractability of $\MinO_C$. This is because the $\MOT_C$ and $\MinO_C$ problems are \emph{polynomial-time equivalent}---for any cost $C$, and for both exact and approximate solving---because the  converse of this reduction also holds~\citep[\S3]{AltBoi20mot}.
\end{remark}

\begin{remark}[Value vs solution for $\MOT$]\label{rem:val-sol-mot}
	A desirable feature of our hardness results is that they apply regardless of how an $\MOT$ \emph{solution} is computed and (compactly) represented\footnote{
		Indeed, the representations produced by $\MOT$ algorithms often vary: e.g., the solution is polynomially-sparse for the Ellipsoid and Multiplicative Weights algorithms; and is fully dense but has a polynomial-size representation which supports certain efficient operations for the Sinkhorn algorithm. See~\citep{AltBoi20mot} for details.
	}. This is because we show hardness for (approximately) computing the optimal \emph{value} of $\MOT$.
\end{remark}

\begin{remark}[Differences from classical LP theory]\label{rem:min}
	The intuition behind the $\MinO$ problem is that it is the feasibility problem for the dual LP to $\MOT$; see the preliminaries section.
	However, it should be emphasized that our reductions rely on the particular structure of the LP defining $\MOT$, and do \emph{not} hold for a general LP and its dual feasibility oracle~\citep{GLSbook}. Moreover, our approximate reduction is even further from the purview of classical LP theory since it can be used to prove hardness of approximating to polynomially small error rather than exponentially small error. 
\end{remark}

\begin{remark}[$p = 0$]
	The $\MinO_C(0)$ problem is to compute the minimum entry of the tensor $C$. Thus, as a special case of our reductions, it follows that if (approximately) computing the minimum entry of $C$ is $\NP$-hard, then so is (approximately) computing $\MOT_C$. In fact, in our applications in \S\ref{sec:applications:lr}, \S\ref{sec:applications:pairwise}, and \S\ref{sec:applications:repulsive}, we prove intractability of $\MinO_C$---and thus of $\MOT_C$---by showing intractability for this ``simple'' case $p = 0$. However, we mention that in general one cannot restrict only to the case $p = 0$: In \S\ref{sec:necessity}, we give a concrete example where $\MinO_C$ is tractable for $p = 0$ but not general $p$. 
\label{rem:pequalszero}
\end{remark}

\subsubsection{Applications}

\paragraph*{Low-rank costs.} In \S\ref{sec:applications:lr}, we demonstrate this toolbox on $\MOT$ problems with low-rank cost tensors given in factored form. Recent algorithmic work has shown that such $\MOT$ problems can be solved to arbitrary precision $\eps > 0$ in $\poly(n,k,\Cmax/\eps)$ time for any \emph{fixed} rank $r$~\citep{AltBoi20mot}. However, this algorithm's dependence on $r$ is exponential, and it is a natural question whether such $\MOT$ problems can be solved in time that is also polynomial in $r$. We show that unless $\P = \NP$, the answer is no. Moreover, our hardness result extends even to approximate computation. This provides a converse to the aforementioned algorithmic result.

\paragraph*{Pairwise-interaction costs.} In \S\ref{sec:applications:pairwise}, we consider $\MOT$ problems with costs $C$ that decompose into sums of pairwise interactions
\begin{align}
C_{j_1,\dots,j_k} = 
\sum_{1 \leq i < i' \leq k} g_{i,i'}(j_i,j_{i'})
\label{eq:C-pairwise}
\end{align}
for some functions $g_{i,i'} : [n] \times [n] \to \RR$. This cost structure appears in many applications; for instance in Wasserstein barycenters~\citep{anderes2016discrete} and the $\MOT$ relaxation of Density Functional Theory~\citep{CotFriKlu13,ButDepGor12}. Although these costs have $\poly(n,k)$-size implicit representations, we show that this cost structure alone is not sufficient for solving $\MOT$ in polynomial time.

\par One implication of this $\NP$-hardness result is a converse to the algorithmic result of~\citep{AltBoi20mot} which shows that $\MOT$ problems can be solved in $\poly(n,k)$ time for costs $C$ which are decomposable into local interactions of low treewidth. This is a converse because the pairwise-interactions structure~\eqref{eq:C-pairwise} also falls under the framework of $\MOT$ costs that decompose into local interactions, but has high treewidth.

\paragraph*{Repulsive costs.} In \S\ref{sec:applications:repulsive}, we consider $\MOT$ problems with ``repulsive costs''. Informally, these are costs $C_{j_1, \dots, j_k}$ which encourage diversity between the indices $j_1, \dots j_k$; we refer the reader to the nice survey~\citep{di2017optimal} for a detailed discussion of such $\MOT$ problems and their many applications. We provide evidence that repulsive costs lead to intractable $\MOT$ problems by proving that several such $\MOT$ problems of interest are NP-hard to solve. Again, our hardness results extend even to approximate computation.
\par Specifically, in \S\ref{ssec:repulsive:det} we show this for $\MOT$ problems with the determinantal cost studied in~\citep{di2017optimal,CarNaz08}, and in \S\ref{ssec:repulsive:dft} we show this for the popular $\MOT$ formulation of Density Functional Theory~\citep{CotFriKlu13,ButDepGor12,BenCarNen16} with the Coulomb-Buckingham potential. Additionally, in \S\ref{ssec:repulsive:submod}, we observe that the classical problem of evaluating the convex envelope of a discrete function is an instance of \MOT, and we leverage this connection to point out that $\MOT$ is $\NP$-hard to approximate with supermodular costs, yet tractable with submodular costs. This dichotomy provides further evidence for the intractability of repulsive costs, since the intractable former problem has a ``repulsive'' cost, whereas the tractable latter problem has an ``attractive'' cost. 

\par To our knowledge, these are the first results that rigorously demonstrate intractability of $\MOT$ problems with repulsive costs. This provides the first step towards explaining why---despite a rapidly growing literature---there has been a lack of progress in developing polynomial-time algorithms with provable guarantees for many $\MOT$ problems with repulsive costs.

\subsection{Related work}\label{ssec:intro:prev}

\paragraph*{Algorithms for $\MOT$.} The many applications of $\MOT$ throughout data science, mathematics, and the sciences at large have motivated a rapidly growing literature around developing efficient algorithms for $\MOT$. The algorithms in this literature can be roughly divided into two categories. 

\par The first category consists of $\MOT$ algorithms which work for generic ``unstructured'' costs. While these algorithms work for any $\MOT$ problem, they inevitably \emph{cannot} have polynomial runtime in $n$ and $k$ since they read all $n^k$ entries of the cost tensor. A simple such algorithm is to solve $\MOT$ using an out-of-the-box LP solver; this has $n^{\Theta(k)}$ runtime. An alternative popular algorithm is the natural multimarginal generalization of the Sinkhorn scaling algorithm, which similarly has $n^{\Theta(k)}$ runtime but can be faster than out-of-the-box LP solvers in practice, see e.g.,~\citep{BenCarCut15,BenCarDi18,BenCarNen16,BenCarNen19,Nen16,LinHoJor19,Fri20,tupitsa2020multimarginal} among many others. The exponential runtime dependence on $n$ and $k$ prohibits these algorithms from being usable beyond very small values of $n$ and $k$. For instance even $n=k=10$, say, is at the scalability limits of these algorithms.

\par The second category consists of $\MOT$ algorithms that can run in $\poly(n,k)$ time\footnote{Or $\poly(n,k,\Cmax/\eps)$ time for $\eps$-approximate solutions.} for $\MOT$ problems with certain ``structured'' costs that have $\poly(n,k)$-sized implicit representations. This line of work includes for instance the algorithms mentioned earlier in the introduction, namely computing generalized Euler flows~\citep{BenCarCut15,AltBoi20mot}, computing low-dimensional Wasserstein barycenters~\citep{CarObeOud15,AltBoi20bary}, solving $\MOT$ problems with tree-structured costs~\citep{h20tree}, and solving $\MOT$ problems with decomposable costs~\citep{AltBoi20mot}. The purpose of this paper is to understand the fundamental limitations of this line of work.

\paragraph*{Connection to fractional hypergraph matching and complexity of sparse solutions.} Deciding whether $\MOT$ has a solution of sparsity exactly $n$ is well-known to be $\NP$-hard. This $\NP$-hardness holds even in the special case where the number of marginals $k=3$, the cost tensor $C$ has all $\{0,1\}$ entries, and the marginals are uniform $\mu_i = \bone_n/n$. This is because in this special case, $\MOT$ is the natural convex relaxation of the $\NP$-hard $k$-partite matching problem~\citep{karp1972reducibility,garey1979computers}, and in particular $n$-sparse solutions to this $\MOT$ problem are in correspondence with optimal solutions to this $\NP$-hard problem. The $\NP$-hardness of finding the sparsest $\MOT$ solution further extends to ``structured'' $\MOT$ problems whose costs have $\poly(n,k)$-size implicit representations, e.g., the Wasserstein barycenter problem~\citep{borgwardt2019computational}.

\par However, it is important to clarify that finding an $\MOT$ solution which is polynomially sparse (albeit perhaps not the \emph{sparsest}) is \emph{not} necessarily $\NP$-hard. Indeed, sparse such solutions can be computed in polynomial-time for e.g., low-dimensional Wasserstein barycenters~\citep{AltBoi20bary} and $\MOT$ problems with decomposable costs~\citep{AltBoi20mot}. 

\par We emphasize that a desirable property of our hardness results is that they entirely bypass this discussion about whether computing sparse solutions is tractable. This is because our results show that it is $\NP$-hard to even compute the \emph{value} of certain structured $\MOT$ problems (see Remark~\ref{rem:val-sol-mot}), which is clearly a stronger statement than $\NP$-hardness of computing a sparse solution.

\paragraph*{Other related work.} We mention two tangentially related bodies of work in passing. First, the transportation polytope (a.k.a., the constraint set in the $\MOT$ problem) is an object of significant interest in discrete geometry and combinatorics, see e.g.,~\citep{yemelicher1984polytopes,de2014combinatorics} and the references within. Second, linear programming problems over exponentially-sized joint probability distributions appear in various fields such as game theory~\citep{pap08} and variational inference~\citep{wainwright2003variational}. However, it is important to note that the complexity of these linear programming problems is heavily affected by the specific linear constraints, which often differ between problems in different fields.

\section{Preliminaries}\label{sec:prelim}

\paragraph*{Notation.} 
The $k$-fold tensor product space $\R^n \otimes \cdots \otimes \R^n$ is denoted by $\Rntk$, and similarly for $\Rpntk$. The set $\{1, \dots, n\}$ is denoted by $[n]$.  The $i$-th marginal of a tensor $P \in \Rntk$ is the vector $m_i(P) \in \R^n$ with $j$-th entry $ \sum_{j_1,\dots,j_{i-1},j,j_{i+1}, \dots, j_k} P_{j_1,\dots,j_{i-1},j,j_{i+1}, \dots, j_k}$,
for $i \in [k]$ and $j \in [n]$. In this notation, the transportation polytope in~\eqref{MOT} is
	$\Coup = \{ P \in \Rpntk : m_i(P) = \mu_i,\; \forall i \in  [k] \}$.
For shorthand, we often denote an index $(j_1,\dots,j_k)$ by $\jvec$.  For a tensor $C \in \Rntk$, we denote the maximum absolute value of its entries by $C_{\max} = \max_{\jvec} |C_{\jvec}|$. For shorthand, we write $\poly(t_1,\dots,t_m)$ to denote a function that grows at most polynomially fast in those parameters.

\paragraph*{MOT dual.} The dual LP to~\eqref{MOT-D} is
\begin{align}
	& \max_{p_1,\ldots,p_k \in \Rn} \sum_{i=1}^k \langle p_i, \mu_i \rangle \tag{MOT-D}
	\;\;\quad \text{subject to} \;\;\quad  
	C_{\jvec} - \sum_{i=1}^k [p_i]_{j_i} \geq 0, \;\; \forall \jvec \in [n]^k.
	\label{MOT-D} 
\end{align}
Observe that $p = (p_1, \dots, p_k) \in \R^{n \times k}$ is feasible for~\eqref{MOT-D} if and only if the problem $\MinO_C(p)$ has non-negative value. This is the connection between the problem $\MinO_C$ and the feasibility oracle for~\eqref{MOT-D} alluded to in Remark~\ref{rem:min}.

\paragraph*{Bit complexity.} Throughout, we assume for simplicity of exposition that all entries of the cost $C$ and the weights $p \in \RR^{nk}$ inputted in the $\MinO_C$ problem have bit complexity at most $\poly(n,k)$. This implies that the distributions $\mu$ on which the $\MOT_C$ oracle is queried in Theorems~\ref{thm:main:exact} and \ref{thm:main:approx} also have polynomial bit complexity. The general case is a straightforward extension.

\paragraph*{Computational complexity.} $\BPP$ is the class of problems solvable by polynomial-time randomized algorithms with error probability that is $< 1/3$ (or equivalently, any constant less than $1/2$). 
The statement ``$\NP \not\subset \BPP$'' is a standard assumption in computational complexity and is the \textit{randomized} version of $\P \neq \NP$, i.e., that $\NP$-hard problems do not have polynomial-time randomized algorithms.

\section{Reducing $\MinO$ to $\MOT$}\label{sec:reduction}

Here we present the reduction toolkit overviewed in \S\ref{ssec:intro:cont}. Specifically, we show the following two reductions from $\MinO_C$ to $\MOT_C$. The first reduction is used for proving $\NP$-hardness of exactly solving $\MOT_C$, and the second reduction is used for proving inapproximability. These two reductions are incomparable.

\begin{theorem}[Exact reduction]\label{thm:main:exact}
There is a deterministic algorithm that, given access to an oracle solving $\MOT_C$ and weights $p \in \R^{n \times k}$, solves $\MinO_C(p)$ in $\poly(n,k)$ time and oracle queries.
\end{theorem}

\begin{theorem}[Approximate reduction]\label{thm:main:approx}
	There is a randomized algorithm that, given $\eps > 0$, access to an oracle solving $\MOT_C$ to additive accuracy $\eps$, and weights $p  \in \R^{n \times k}$, solves $\MinO_C(p)$ up to $\eps \cdot  \poly(n,k)$ additive accuracy with probability $2/3$ in $\poly(n,k,\tfrac{\Cmax}{\eps})$ time and oracle queries. 
\end{theorem}

We make two remarks in passing about these theorems. First, they hold unchanged if the $\MinO_C$ problem is modified to require computing the minimizing tuple rather than the minimum value in~\eqref{eq:def:min}. This is because these two problems are polynomial-time equivalent~\citep[Appendix A.1]{AltBoi20mot}. Second, the inapproximability reduction is probabilistic\footnote{It is an interesting question whether the approximate reduction in Theorem~\ref{thm:main:approx} can be de-randomized. This would enable showing our inapproximability results under the assumption $\P \neq \NP$ rather than $\NP \not\subset \BPP$.}, and thus shows inapproximability under the standard complexity assumption $\NP \not\subset \BPP$, which is informally the stronger ``randomized version'' of $\P \neq \NP$.

\subsection{Proof overview}\label{ssec:reduction:overview}

As described briefly in \S\ref{ssec:intro:cont}, the main idea is to reduce $\MinO_C$ to a convex optimization problem for which the objective function can be evaluated by solving an auxiliary $\MOT_C$ problem. Formally, embed $[n]^k$ into $\R^{nk}$ via 
\[
\phi((j_1,\dots,j_k)) := (e_{j_1},\dots,e_{j_k})^T,
\]
where $e_{j}$ denotes the vector in $\R^n$ with a $1$ on entry $j$, and $0$'s on all other entries. Then $\MinO_C$ is equivalent to minimizing $f(x) := C_{\phi^{-1}(x)} - \langle p, x \rangle$ over the discrete set 
\[
\cS := \phi([n]^k) = \{x = (x_1, \dots, x_k)^T \in \{0,1\}^{nk} \, : \, \|x_1\|_1 = \dots = \|x_k\|_1 = 1 \},
\]
where here we abuse notation slightly by viewing $p = (p_1, \dots, p_k) \in \R^{n \times k}$ as the vector in $\R^{nk}$ formed by concatenating its columns $p_1, \dots, p_k$. Let $F : \R^{nk} \to \R$ denote the convex envelope of $f$, i.e., the pointwise largest convex function over $\R^{nk}$ that is pointwise below $f$ on the domain $\cS$ of $f$. By explicitly computing $F$ as the Fenchel bi-conjugate of $f$, we obtain the following two useful characterizations of $F$ (proved in \S\ref{ssec:reduction:pf}). Below, let $\Delta_{\cS}$ denote the set of probability distributions over $\cS$.

\begin{lemma}[$\MOT_C$ is convex envelope of $\MinO_C$]\label{lem:mot-cvx}
	For all $\mu = (\mu_1,\dots,\mu_k) \in (\Delta_n)^k$, 
	\begin{align}
		F(\mu)
		&= \min_{D \in \Delta_\cS \text{ s.t. } \E_{x \sim D} x = \mu} \E_{x \sim D} f(x)
		\label{eq:mot-cvx:1}
		\\ &= 
		-\langle \mu,p \rangle + 
		\mathsf{MOT}_C(\mu).
		\label{eq:mot-cvx:2}
	\end{align}
\end{lemma}

The first representation~\eqref{eq:mot-cvx:1} of $F$ gives a Choquet integral representation of $F$ in terms of $f$. Importantly, it implies that in order to (approximately) minimize $f$ over its discrete domain $\cS$---i.e., solve the (approximate) $\MinO_C$ oracle---it suffices to (approximately) minimize $F$ over its continuous domain, namely the convex hull $\conv(\cS) = (\Delta_n)^k$. See Corollary~\ref{cor:mot-cvx}.
\par Now to (approximately) minimize $F$, we appeal to algorithmic results from zero-th order convex optimization since the second representation~\eqref{eq:mot-cvx:2} of $F$ shows that (approximately) evaluating $F$ amounts to (approximately) solving an auxiliary $\MOT_C$ problem. Specifically, we show how to implement the zero-th order optimization using the Ellipsoid algorithm in the case of exact oracle evaluations, and otherwise using the recent results~\citep{Ris16,Bel15} on zero-th order optimization of approximately convex functions in the case of approximate oracle evaluations.

\begin{remark}[Connection to submodular optimization]
	This proof is inspired by the classical idea of minimizing a submodular function by minimizing its Lov\'asz extension~\citep{GroLovSch81}, which is a special case of Theorem~\ref{thm:main:exact} for $n=2$ and submodular costs $C$. In fact, in light of the equivalence between $\MOT_C$ and the Lov\'asz extension in that special case (described in \S\ref{ssec:repulsive:submod}), this is arguably the appropriate generalization thereof to optimizing general discrete functions over general ground sets of size $n \geq 2$.
\end{remark}

\subsection{Proofs}\label{ssec:reduction:pf}

\begin{proof}[Proof of Lemma~\ref{lem:mot-cvx}]
	The convex envelope $F$ of $f$ is equal to the Fenchel bi-conjugate $f^{**}$ of $f$~\citep{Rocbook}. The Fenchel conjugate of $f$ is $f^* : \R^{nk} \to \R$ where $f^*(y) = \max_{x \in \cS} \langle x,y \rangle - f(x) = \max_{x \in \cS} \langle x,y+p\rangle - C_{\phi^{-1}(x)}$. Thus the Fenchel bi-conjugate $f^{**}$ is
	\begin{align*}
		F(\mu)
		=
		f^{**}(\mu)
		=
		\max_{y \in \R^{nk}} \langle y, \mu \rangle - f^*(y)
		= 
		\max_{y \in \R^{nk}} \min_{x \in \cS} \langle y, \mu - x \rangle - \langle x, p \rangle + C_{\phi^{-1}(x)}.
	\end{align*}
	By performing a convex relaxation over the inner minimization (to distributions $D$ over the set $\cS$) and then invoking LP strong duality, we obtain
	\begin{align*}
		F(\mu)
		=
		\min_{D \in \Delta_{\cS}} \max_{y \in \R^{nk}} \langle y, \mu - \E_{x \sim D} x \rangle - \langle \E_{x \sim D} x, p \rangle + \E_{x \sim D} C_{\phi^{-1}(x)} 
	\end{align*}
	Note that the inner maximization over $y$ has unbounded cost $+\infty$ unless $\E_{x \sim D} x = \mu$. Thus
	\begin{align*}
		F(\mu)
		=
		- \langle \mu,p \rangle
		+
		\min_{D \in \Delta_{\cS} \; \text{s.t.} \; \E_{x \sim D} x = \mu} \E_{x \sim D} C_{\phi^{-1}(x)},
	\end{align*}
	This proves~\eqref{eq:mot-cvx:1} by definition of $f$. Now~\eqref{eq:mot-cvx:2} follows since distributions $D$ over $\cS$ with expectation $\mu$ are in correspondence with joint distributions $P \in \Coup$, and under this correspondence $\E_{x \sim D} C_{\phi^{-1}(x)}$ simply amounts to $\langle P, C \rangle$.	
\end{proof}

\begin{cor}[Minimizing $F$ suffices for minimizing $f$]\label{cor:mot-cvx}
	The minimum value of $F$ over $(\Delta_n)^k$ is equal to the minimum value of $f$ over $S$. 
\end{cor}
\begin{proof}
	By the Choquet representation~\eqref{eq:mot-cvx:1} of $F$ in Lemma~\ref{lem:mot-cvx}, the set of minimizers of $F$ over $(\Delta_n)^k$ is equal to the convex hull of the minimizers of $f$ over $S$.
\end{proof}

\subsubsection{Hardness of computation} 

\begin{proof}[Proof of Theorem~\ref{thm:main:exact}]
	By Corollary~\ref{cor:mot-cvx}, it suffices to minimize $F$ over $(\Delta_n)^k$ in the desired runtime. 
	\par To this end, we claim that $F$ is the maximum of a finite number of linear functions, each of which has polynomial encoding length in the sense of~\citep[\S6.5]{GLSbook}. To show this statement, it suffices to show the same statement for the function $\mu \mapsto \MOT_C(\mu)$ by the representation~\eqref{eq:mot-cvx:2} of $F$ in Lemma~\ref{lem:mot-cvx} that equates $F$ to a linear function plus $\MOT_C$. This latter statement follows by the dual MOT formulation~\eqref{MOT-D} and a standard LP argument. Specifically, the function $\mu \mapsto \MOT_C(\mu)$ is a linear function in $\mu$ with finitely many pieces, one for each vertex of the polyhedral feasible set defining~\eqref{MOT-D} by the Minkowski-Weyl Theorem; and furthermore, the vertices are solutions to linear systems in the constraints, and thus have polynomial bit complexity by Cramer's Theorem.
	\par Therefore, since $(\Delta_n)^k$ is a ``well-described polyhedron'' in the sense of~\citep[Definition 6.2.2]{GLSbook}, we may apply the Ellipsoid algorithm in~\citep[Theorem 6.5.19]{GLSbook}. That theorem shows that $F$ can be minimized over $(\Delta_n)^k$ using polynomially many evaluations of $F$ and polynomial additional processing time. By appealing again to the representation~\eqref{eq:mot-cvx:2} of $F$ in Lemma~\ref{lem:mot-cvx}, each evaluation of $F$ can be performed via a single $\MOT_C$ computation and polynomial additional processing time.
\end{proof}

\subsubsection{Hardness of approximation}

\begin{proof}[Proof of Theorem~\ref{thm:main:approx}]
	By Corollary~\ref{cor:mot-cvx}, it suffices to compute the minimum value of $F$ over $(\Delta_n)^k$ to additive accuracy $\Theta(\eps nk)$. By the representation of $F$ in~\eqref{eq:mot-cvx:2}, this amounts to approximately computing
	\begin{align}
		\min_{\mu \in (\Delta_n)^k} 
		\MOT_C(\mu)
		- \langle \mu, p \rangle.
		\label{eq:thm-amin-to-amot}
	\end{align} 
to that accuracy.
Note that a query to the oracle computing $\MOT_C$ to $\eps$ accuracy (plus polynomial-time additional computation) computes the objective function in~\eqref{eq:thm-amin-to-amot} to $\eps$ additive accuracy. Therefore, this is an instance of the zero-th order optimization problem for approximately convex functions studied in~\citep{Bel15,Ris16}. The claimed runtime and approximation accuracy follow from their results once we check that $F$ has polynomial Lipschitz parameter, done next (we give a tighter bound than needed since it may be of independent interest).
\end{proof}

\begin{lemma}[$\ell_1$-Lipschitzness of $\MOT_C$ w.r.t. marginals]\label{lem:mot-lip}
	The function $\mu \mapsto \MOT_C(\mu)$ on $(\Delta_n)^k$ is Lipschitz with respect to the entrywise $\ell_1$ norm with parameter $2\Cmax$. 
\end{lemma}
\begin{proof}
	Let $\mu, \mu' \in (\Delta_n)^k$. By symmetry, it suffices to show that
	\begin{align*}
		\min_{P' \in \Coupp} \langle P', C \rangle
		\leq 
		\min_{P\in \Coup} \langle P, C \rangle 
		 + 2 \Cmax \|\mu- \mu'\|_1.
	\end{align*}
Let $P^*$ be an optimal solution for the optimization over $\Coup$. By the rounding algorithm in~\citep{LinHoJor19}, there exists $\hat{P} \in \Coupp$ such that the entrywise $\ell_1$ norm $\|\hat{P} - P^*\|_1 \leq 2 \|\mu - \mu'\|_1$. Thus 
	\begin{align*}
		\min_{P' \in \Coupp} \langle P', C \rangle
		\leq 
		\langle \hat{P}, C \rangle
		= 
		\langle P^*, C \rangle + \langle \hat{P} - P^*, C \rangle.
	\end{align*}
	By construction of $P^*$, the first term $\langle P^*, C \rangle = \min_{P\in \Coup} \langle P, C \rangle $. By H\"older's inequality and the construction of $\hat{P}$, the second term $\langle \hat{P} - P^*, C \rangle \leq \Cmax \|\hat{P} - P^*\|_1  \leq 2 \Cmax 
	\|\mu - \mu'\|_1$. 
\end{proof}

\section{Application: costs with super-constant rank}\label{sec:applications:lr}

Recent work has given a polynomial time algorithm for approximate $\MOT$ when the cost is a constant-rank tensor given in factored form \citep{AltBoi20mot}. A natural algorithmic question is whether the dependence on the rank can be improved: is there an algorithm whose runtime is simultaneously polynomial in $n$, $k$, and the rank $r$? Here we show that, under standard complexity theory assumptions, the answer is no. Our result provides a converse to \citep{AltBoi20mot}, and justifies the constant-rank regime studied in \citep{AltBoi20mot}.

\begin{prop}[Hardness of $\MOT$ for low-rank costs]\label{prop:lr-hard}
	Assuming $\P\neq \NP$, there does not exist a  $\poly(n,k,r)$-time deterministic algorithm for solving $\MOT_C$ for costs $C$ given by a rank-$r$ factorization.
\end{prop}

Our impossibility result further extends to approximate computation.

\begin{prop}[Hardness of approximate $\MOT$ for low-rank costs]\label{prop:lr-hard-approx}
	Assuming $\NP\not\subset \BPP$, there does not exist a $\poly(n,k,r,\tfrac{\Cmax}{\eps})$-time randomized algorithm for approximating $\MOT_C$ to $\eps$ additive accuracy for costs $C$ given by a rank-$r$ factorization.
\end{prop}

The proof encodes the hard problem of finding a large clique in a $k$-partite graph as an instance of $\MOT_C$ in which $C$ has an explicit low-rank factorization. We define the following notation: for a $k$-partite graph $G$ on $nk$ vertices $v_{i,j}$ for $i \in [k]$ and $j \in [n]$, let $T_G \in \Rntk$ denote the tensor with $(j_1,\dots,j_k)$-th entry equal to the number of edges in the induced subgraph of $G$ with vertices $\{v_{1,j_1}, \dots, v_{k,j_k}\}$. 

\begin{lemma}[$T_G$ is low-rank]\label{lem:lr-Tg}
	For any $k$-partite graph $G$ on $nk$ vertices, $\rank(T_G) \leq n^2k^2$. Moreover, a factorization of $T_G$ with this rank is computable from $G$ in $\poly(n,k)$ time.
\end{lemma}
\begin{proof}
	Consider an edge $(v_{i,j_i}, v_{i',j_{i'}})$ between partitions $i,i' \in [k]$. Consider the rank-$1$ tensor formed by the outer product of the indicator vectors $e_{j_i}$ and $e_{j_{i'}}$ on respective slices $i$ and $i'$, and the all-ones vector $\bone_n$ on all other slices $\ell \in [k] \setminus \{i,i'\}$. This tensor takes value $1$ on all tuples in $[n]^k$ with $i$-th coordinate $j_i$ and $i'$-th coordinate $j_{i'}$, and takes value $0$ elsewhere. Summing up such a rank-$1$ tensor for each edge of $G$---of which there are at most $(nk)^2$---yields the desired factorization.
\end{proof}

\begin{lemma}[Hardness of $\MinO$ for low-rank costs]\label{lem:lr-hard-min}
	Assuming $\P \neq \NP$, there is no $\poly(n,k,r)$-time deterministic algorithm for solving $\MinO_C$ for costs $C$ given by a rank-$r$ factorization. Moreover, assuming $\NP\not\subset \BPP$, 	there is $\poly(n,k,r,\tfrac{\Cmax}{\eps})$-time randomized algorithm for $\eps$-approximate additive computation.
\end{lemma}
\begin{proof}
	Deciding whether there exists a $k$-clique in a $k$-partite graph $G$ on $nk$ vertices is $\NP$-hard. This problem reduces to computing the maximal entry in $T_G$, which is equivalent to solving $\MinO_C(0)$ for $C=-T_G$. The first statement then follows since a low-rank factorization of $-T_G$ can be found in $\poly(n,k)$ time by Lemma~\ref{lem:lr-Tg}. For the second statement, note that since the entries of $-T_G$ are integral, it is also $\NP$-hard to solve $\MinO_C(0)$ to additive error $\Cmax / 10k^2 \leq 0.1$.
\end{proof}

\begin{proof}[Proof of Proposition~\ref{prop:lr-hard}]
By Theorem~\ref{thm:main:exact}, a $\poly(n,k,r)$-time deterministic algorithm for $\MOT_C$ on rank-$r$ costs implies a $\poly(n,k,r)$-time deterministic algorithm for $\MinO_C$ on rank-$r$ costs. Assuming $\P \neq \NP$, this contradicts Lemma~\ref{lem:lr-hard-min}.
\end{proof}

\begin{proof}[Proof of Proposition~\ref{prop:lr-hard-approx}]
By Theorem~\ref{thm:main:approx}, a $\poly(n,k,r,\tfrac{\Cmax}{\eps})$-time randomized algorithm for $\MOT_C$ on rank-$r$ costs $C$ implies a $\poly(n,k,r,\tfrac{\Cmax}{\eps})$-time randomized algorithm for $\MinO_C$ on rank-$r$ costs $C$. Assuming $\NP \not\subset \BPP$, this contradicts Lemma~\ref{lem:lr-hard-min}.
\end{proof}

\section{Application: costs with full pairwise interactions}\label{sec:applications:pairwise}

Many studied $\MOT$ costs, such as the Wasserstein barycenter cost and Coulomb cost, have the following structure: they decompose into a sum of pairwise interactions, as
\begin{align}
	C_{j_1,\ldots,j_k} = \sum_{1 \leq i < i' \leq k} g_{i,i'}(j_i,j_{i'})
	\label{eq:pairwise}
\end{align}
for some functions $g_{i,i'} : [n] \times [n] \to \RR$. This decomposability structure allows for a polynomial-size implicit representation of the cost tensor. It is a natural question whether this generic structure can be exploited to obtain polynomial-time algorithms for $\MOT$. We show that the answer is no: there are $\MOT$ costs that are decomposable into pairwise interactions, but are $\NP$-hard to solve.

\begin{prop}[Hardness of $\MOT$ for pairwise-decomposable costs]\label{prop:pd-hard}
	Assuming $\P\neq \NP$, there does not exist a  $\poly(n,k)$-time deterministic algorithm for solving $\MOT_C$ for costs $C$ of the form~\eqref{eq:pairwise}.
\end{prop}

Our impossibility result further extends to approximate computation.

\begin{prop}[Hardness of approximate $\MOT$ for pairwise-decomposable costs]\label{prop:pd-hard-approx}
	Assuming $\NP\not\subset \BPP$, there does not exist a $\poly(n,k,\tfrac{\Cmax}{\eps})$-time randomized algorithm for approximating $\MOT_C$ to $\eps$ additive accuracy for costs $C$ of the form~\eqref{eq:pairwise}.
\end{prop}

\begin{proof}[Proof of Propositions~\ref{prop:pd-hard} and \ref{prop:pd-hard-approx}] The proofs of Propositions~\ref{prop:pd-hard} and \ref{prop:pd-hard-approx} are the same as the proofs of Propositions~\ref{prop:lr-hard} and \ref{prop:lr-hard-approx} using the fact that for any graph $G = (V,E)$, the tensor $-T_G$ can be written as a sum of pairwise interactions: $(-T_G)_{j_1,\ldots,j_k} = \sum_{1 \leq i < i' \leq k} -\mathds{1}[(v_{i,j_i}, v_{i',j_{i'}}) \in E]$. 
\end{proof}

Propositions~\ref{prop:pd-hard} and \ref{prop:pd-hard-approx} provides converses to the result of \citep{AltBoi20mot}. Specifically, \citep[\S4]{AltBoi20mot}, considers $\MOT$ costs $C$ that decompose into local interactions as $C_{j_1,\ldots,j_k} = \sum_{S \in \mathcal{S}} g_S(\{j_i\}_{i \in S})$, and gives a polynomial-time algorithm in the case that the graph with vertices $[k]$ and edges $\{(i,i') : i,i' \in S \mbox{ for some } S \in \mathcal{S}\}$ has constant treewidth. Conversely, our hardness results in Propositions~\ref{prop:lr-hard} and \ref{prop:lr-hard-approx} show that bounded treewidth is necessary for polynomial-time algorithms. This is because costs of the form~\eqref{eq:pairwise} fall under the framework of decomposable costs in~\citep{AltBoi20mot} with non-constant treewidth of size $k-1$.

\section{Application: repulsive costs}\label{sec:applications:repulsive}

In this section, we investigate several $\MOT$ problems with repulsive costs that are of interest in the literature. 
We prove intractability results that clarify why---despite a growing literature (see, e.g., the survey~\citep{di2017optimal} and the references within)---these problems have resisted algorithmic progress.

\subsection{Determinantal cost}\label{ssec:repulsive:det}

A repulsive cost of interest in the $\MOT$ literature is the determinant cost (e.g.,~\citep{di2017optimal,CarNaz08}). This cost is given by:
\begin{equation}\label{eq:harddetcost}C_{j_1,\ldots,j_k} = -|\det(x_{j_1},\ldots,x_{j_k})|,\end{equation} where $x_1,\ldots,x_n \in \RR^k$ and $\det(x_{j_1},\ldots,x_{j_k})$ is the determinant of the $k \times k$ matrix whose columns are $x_{j_1},\ldots,x_{j_k}$. This is a repulsive cost in the sense that tuples with ``similar'' vectors are penalized with higher cost, see the survey~\citep{di2017optimal}. We prove that the $\MOT$ problem with this cost is $\NP$-hard. For convenience of notation, we think of the marginal distributions $\mu_1,\ldots,\mu_k$ as distributions in the simplex $\Delta_n$, and write $[\mu_i]_j$ to mean the mass of $\mu_i$ on $x_j$.

\begin{prop}[Hardness of $\MOT$ with determinant cost]\label{prop:hard:det}
Assuming $\P \neq \NP$, then there is no $\poly(n,k)$-time algorithm that given $x_1,\ldots,x_n \in \RR^k$ solves $\MOT_C$ for the cost $C$ in \eqref{eq:harddetcost}.
\end{prop}
\begin{proof}
	By Theorem~\ref{thm:main:exact}, it suffices to prove that the $\MinO_C$ problem is $\NP$-hard. We show this is true even if the input weights $p$ are identically $0$: in this case the $\MinO_C$ problem is to compute $\min_{\jvec} C_{\jvec} = -\max_{\jvec} |\det(x_{j_1},\dots,x_{j_k})|$ given $x_1,\ldots,x_n \in \RR^k$. This is $\NP$-hard by \citep{papadimitriou1984largest}.
\end{proof}

Rather than show additive inapproximability of $\MOT$ with determinant costs, we consider log-determinant costs since additive error on the logarithmic scale amounts to multiplicative error on the natural scale, which is more standard in the combinatorial-optimization literature on determinant maximization. Below, we show inapproximability of $\MOT$ with such log-determinant costs. Note that for technical reasons we upper-bound the cost at $0$ to avoid unbounded costs for tuples with null determinant:
\begin{equation}
	C_{j_1,\ldots,j_k} = \min(0,-\log |\det(x_{j_1},\ldots,x_{j_k})|). \label{eq:hardlogdetcost}
\end{equation}

\begin{prop}[Approximation hardness of $\MOT$ with log-determinant cost]\label{prop:hard:det:appx}
Assuming $\NP\not\subset \BPP$, then there is no $\poly(n,k,\Cmax/\eps)$ time algorithm that given $x_1,\ldots,x_n \in \RR^k$ approximates $\MOT_C$ to $\eps$ additive accuracy for the cost $C$ in \eqref{eq:hardlogdetcost}.
\end{prop}
\begin{proof}
	Let $x_1,\ldots,x_n \in \ZZ^k$ have $\poly(n,k)$ bits each. It is known to be $\NP$-hard to approximate $\min_{j_1,\ldots,j_k} - \log |\det(x_{j_1},\ldots,x_{j_k})|$ to within additive error $0.0001$ \citep[Theorem 3.2]{summa2014largest}. Since $x_1,\ldots,x_n$ span $\R^k$ without loss of generality, this is equivalent to approximating $\min_{j_1,\ldots,j_k} C_{j_1,\ldots,j_k}$ to within additive error $0.0001$. 
	But by Theorem~\ref{thm:main:approx}, given access to $\MOT_C$ computations with additive accuracy $\Cmax/\poly(n,k)$, we can approximate $\MinO_C(0) = \min_{j_1,\ldots,j_k} C_{j_1,\ldots,j_k}$ to within additive error $0.0001$ in $\poly(n,k)$ randomized time since $\Cmax$ is of $\poly(n,k)$ size here. Hence, assuming $\BPP \not\subset \NP$ there is no $\poly(n,k,\Cmax/\eps)$-time algorithm that solves $\MOT_C$ to accuracy $\eps$.
\end{proof}

\subsection{Supermodular cost}\label{ssec:repulsive:submod}

We now consider $\MOT$ problems given by discrete functions that are either submodular or supermodular (see, e.g.,~\citep{fujishige2005submodular} for definitions). Specifically, consider an $\MOT$ problem with $n = 2$ and a cost $C : \{0,1\}^k \to \RR$ that is submodular or supermodular.\footnote{Here, we index the marginals of the cost tensor $C$ with the ground set $\{0,1\}$, instead of $\{1,2\}$ as we would in the rest of the paper, to match the notational convention for supermodular/submodular functions.} Since $\{0,1\}^k$ corresponds to the power set of $[k]$, the cost $C$ can be equivalently viewed as a set function on subsets $S \subseteq [k]$. The $\MOT$ problem is of the form
\begin{align}
	\min_{P \in \cM(\Ber(x_1), \dots, \Ber(x_k))} \E_{S \sim P} \ C(S)
	\label{eq:supmod-mot}
\end{align} where $x_1, \dots, x_k \in [0,1]$ dictate the marginals $\mu_1, \dots, \mu_k$, and $\Ber(p)$ denotes a Bernoulli distribution taking value $1$ with probability $p$. In words,~\eqref{eq:supmod-mot} is an optimization problem over distributions $P$ on subsets of $[k]$, where the linear cost is the expected value of $C$ with respect to $P$.

\par We prove a dichotomy for $\MOT$ problems with these costs: $\MOT$ is polynomial-time solvable for general submodular costs, but is intractable for general supermodular costs. This aligns with our message that repulsive structure is a source of intractability in $\MOT$, since submodular costs are a prototypical example of ``attractive'' costs, whereas supermodular costs   are often used to model ``repulsive'' costs (see, e.g., \citep{borodin2012max,prasad2014submodular}).

\begin{prop}\label{prop:hard-supmod}
	Consider a function $C : \{0,1\}^k \to \R$ given through oracle access for evaluation, and marginal probabilities $x_1, \dots, x_k \in [0,1]$.
	\begin{itemize}
		\item If $C$ is supermodular, then, assuming $\P \neq \NP$, there is no $\poly(k)$-time algorithm for $\MOT_C$.  Moreover, assuming $\NP \not\subset \BPP$, there is no $\poly(k,\Cmax/\eps)$-time algorithm for computing an $\eps$-approximation.
		\item If $C$ is submodular, then $\MOT_C$ is solvable in $\poly(k)$ time.
	\end{itemize}
\end{prop} 
\begin{proof}	
To show the intractability of computing the $\MOT_C$ problem~\eqref{eq:supmod-mot} with supermodular costs, consider the case in which $C$ is the supermodular function encoding the $\NP$-hard \textsc{Max-Cut} problem for a graph on $k$ vertices (refer to e.g., \citep{feige2011maximizing}). In this case, $C$ is integer-valued, so \textsc{Max-Cut} reduces to approximating $\MinO_C(0) = \min_{S \subseteq \{0,1\}^k} C(S)$ to within, say, $\pm 0.49$ additive error. Thus, $\MinO_C(0)$ is hard to approximate to $\pm 0.49$ error. 
Theorem \ref{thm:main:exact} reduces computing $\MinO_C(0)$ to exactly computing $\MOT_C$, hence exact computation of $\MOT_C$ is also $\NP$-hard. Furthermore, Theorem~\ref{thm:main:approx} uses a randomized algorithm to reduce computing $\MinO_C(0) \pm 0.49$ to $1/\poly(k)$-approximating $\MOT_C$, since the range of $C$ is bounded by $\Cmax \leq k^2$. Therefore, if $\NP \not\subset \BPP$, then there is no $\poly(k,\Cmax/\eps)$-time algorithm for $\eps$-approximation of $\MOT_C$.

On the other hand, if $C$ is submodular, then the problem is tractable. The proof hinges on the observation that the $\MOT_C$ problem~\eqref{eq:supmod-mot} with marginals $\mu_i = \Ber(x_i)$ is equivalent to the LP characterization of evaluating the convex envelope $F : [0,1]^k \to \R$ of $C$ at the point $x = (x_1, \dots, x_k)$~\citep[\S10.1]{GLSbook}. If $C$ is submodular, then $F$ can be evaluated in $O(k)$ evaluations of $C$ and $O(k \log k)$ additional processing time by leveraging the equivalence of $F$ to the Lov\'asz extension of $C$~\citep[\S10.1]{GLSbook}.
\end{proof}

\subsection{Application to Density Functional Theory}\label{ssec:repulsive:dft}

A popular application of $\MOT$ is to formulate a relaxation of the Density Functional Theory problem (DFT) from quantum chemistry.
We refer the reader to~\citep{CotFriKlu13} for an introduction of the $\MOT$ formulation of DFT, and sketch the simplest case below. In the simplest version of the $\MOT$ relaxation, we are given $k$ distributions corresponding to $k$ electron clouds in space, and the objective is to couple the electron clouds in a way that minimizes the expected potential energy of the electron configuration. Suppose the electron clouds are given as distributions $\mu_1,\ldots,\mu_k$ supported on $x_1,\ldots,x_n \in \RR^3$; again, for convenience of notation, we think of $\mu_1,\ldots,\mu_k$ as distributions in the simplex $\Delta_n$, and write $[\mu_i]_j$ to mean the mass of $\mu_i$ on $x_j$. Then, the $\MOT$ relaxation of DFT is to compute a minimum-cost coupling of $\mu_1,\ldots,\mu_k$, with cost given by the Coulomb potential \begin{equation}C_{j_1,\ldots,j_k} = \sum_{1 \leq i < i' \leq k} \frac{1}{\|x_{j_i} - x_{j_{i'}}\|_2}.\label{eq:coulombpot}\end{equation}
This is a repulsive cost that encourages tuples $(j_1,\ldots,j_n) \in [n]^k$ such that $x_{j_1},\ldots,x_{j_n}$ are spread as far as possible, since the Coulomb potential decreases as two electrons move farther apart. Despite significant algorithmic interest, provable polynomial-time algorithms have not yet been found.
We conjecture that in fact solving $\MOT$ with the Coulomb potential is $\NP$-hard.
\begin{conj}\label{conj:coulombhardness}
Assuming $\P \neq \NP$, there is no $\poly(n,k)$-time algorithm solving $\MOT_C$ with the Coulomb potential cost \eqref{eq:coulombpot}.
\end{conj}
In this section, we make progress towards the conjecture by proving hardness of DFT with the related Coulomb-Buckingham potential, which is similar to the Coulomb potential, but has extra energy terms that grow as $1/r^6$ and $\exp(-\Theta(r))$. The Coulomb-Buckingham potential is popular for modeling the structures of ionic crystals~\citep{adamson2020hardness}, and is defined for two particles at distance $r$ with charges $q_1,q_2 \in \{-1,+1\}$ as:
$$U(r,q_1,q_2) = \begin{cases} M, & r = 0 \\ \frac{A_{q_1q_2}}{\exp(B_{q_1q_2} r)} - \frac{C_{q_1q_2}}{r^6} + \frac{q_1q_2}{r}, & r > 0 \end{cases},$$ where $A_{+1}, A_{-1}, B_{+1}, B_{-1}, C_{+1}, C_{-1}$ are constants determining the relative strengths of the terms in the interaction, and $M > 0$ is a large constant (that should be intuitively thought of as infinite) penalizing two ions being in the same place. Given ions with charges $q_j \in \{-1,+1\}$ at positions $x_j \in \RR^3$, the corresponding $\MOT$ cost is given by: \begin{equation}\label{eq:harddftcost} C_{j_1,\ldots,j_k} = \begin{cases} M, & \sum_{i \in [k]} q_{j_i} \neq 0 \\ \sum_{1 \leq i < i' \leq k} U(\|x_{j_i} - x_{j_{i'}}\|_2, q_{j_i}, q_{j_{i'}}), & \sum_{i \in [k]} q_{j_i} = 0 \end{cases}.\end{equation}

\begin{prop}[Hardness of DFT with Coulomb-Buckingham potential]\label{prop:hard-dft}
Assuming $\P \neq \NP$, then there is no $\poly(n,k)$-time algorithm that, given positions $x_1,\ldots,x_n \in \RR^3$, charges $q_1,\ldots,q_n \in \{-1,+1\}$, parameters $A_{\pm 1}, B_{\pm 1}, C_{\pm 1}, M > 0$, and marginals $\mu_1,\ldots,\mu_k \in \Delta_n$, solves $\MOT_C$ with cost $C$ given by \eqref{eq:harddftcost}.
\end{prop}
\begin{proof} 
For the proof, we show $\NP$-hardness even if the inputs $x_1,\ldots,x_n \in \RR^3$ are such that $\min_{1 \leq j' \leq j \leq n} \|x_j - x_{j'}\|_2 \geq 1$, $A_{\pm 1}, B_{\pm 1}, C_{\pm 1} \leq \poly(n,k)$, and $2k^2(2 + A_{+1} + A_{-1} + C_{+1} + C_{-1}) \leq M \leq \poly(n,k)$.
In the parameter regime above, we have $\Cmax = M \leq \poly(n,k)$, so by Theorem~\ref{thm:main:exact} and \ref{thm:main:approx}, it suffices to show that computing $\MinO_C(0)$ is $\NP$-hard.

Furthermore, in the parameter regime above, $\MinO_C(0)$ is equal to the following:
\begin{equation}\min \left\{\frac{1}{2}\sum_{j \in S, j' \in S \sm \{j\}} U(\|x_j - x_{j'}\|_2, q_j, q_{j'}) : S \subset [n], |S| = k, \sum_{j \in S} q_j = 0\right\}\label{eq:modifieddftobjective}\end{equation}
This optimization problem is $\NP$-hard by \citep[Theorem 5]{adamson2020hardness}.
\end{proof}

A similar hardness result also holds for approximate computation, stated next.

\begin{prop}[Approximation hardness of DFT with Coulomb-Buckingham potential]\label{prop:hard-dft-approx}
If $\NP \not\subset \BPP$, there is no $\poly(n,k,\Cmax/\eps)$-time algorithm computing an $\eps$-additive approximation to $\MOT_C$, where $C$ is as in Proposition \ref{prop:hard-dft}.
\end{prop}

The proof of Proposition~\ref{prop:hard-dft-approx} is identical to the proof of Proposition~\ref{prop:hard-dft} once we show that the $\MinO_C(0)$ problem is hard to solve approximately. While~\citep[Theorem 5]{adamson2020hardness} only shows hardness of exactly computing $\MinO_C(0)$, a slightly more careful analysis extends this hardness to approximate computation; details are deferred to the appendix.

\section{Neccesity of dual weights}\label{sec:necessity}
This section fleshes out the details for Remark~\ref{rem:pequalszero}. Namely, in \S\ref{sec:applications:lr}, \S\ref{sec:applications:pairwise}, and \S\ref{sec:applications:repulsive}, we showed that $\MOT_C$ was hard to compute for some family of costs $C$ by proving that $\MinO_C(0)$ was hard to compute. Here, we show that such arguments do not use the full power of Theorems~\ref{thm:main:exact} and \ref{thm:main:approx}:  we construct a family of cost tensors $C$ for which $\MOT_C$ is $\NP$-hard to compute yet $\MinO_C(0)$ is polynomial-time computable.

The cost family is as follows: given a $2$-SAT formula $\phi : \{0,1\}^k \to \{0,1\}$, define \begin{equation}C_{j_1,\ldots,j_k} = -\phi(j_1,\ldots,j_k).\label{eq:2satform}\end{equation}

\begin{prop}
Given a 2-SAT formula $\phi$, it is $\NP$-hard to solve $\MOT_C$ for the cost \eqref{eq:2satform}. However, $\MinO_C(0)$ can be computed in polynomial time.
\end{prop}
\begin{proof}
Observe that $\MinO_C(0) = \min_{j_1,\ldots,j_k} C_{j_1,\ldots,j_k} = -\max_{j_1,\ldots,j_k} \phi(j_1,\ldots,j_k)$ is the satisfiability problem for $\phi$, which can be solved in polynomial-time since $\phi$ is a 2-SAT formula \citep{gusfield1992bounded}.

On the other hand, let $p = (p_1,\ldots, p_k) \in \RR^{2 \times k}$ be given by $p_1 = p_2 = \dots = p_k = (0,-1/(2k)) \in \RR^2$. Then $\MinO_C(p) = \min_{j_1,\ldots,j_k} -\phi(j_1,\ldots,j_k) - \sum_{i=1}^k [p_i]_{j_i} = -[\max_{j_1,\ldots,j_k} \phi(j_1,\ldots,j_k) - \|\vec{j}\|_1 / (2k)]$ solves the problem of finding the minimum weight of a satisfying assignment to $\phi$. This problem is $\NP$-hard~\citep{gusfield1992bounded}, hence $\MinO_C(p)$ is $\NP$-hard. Therefore $\MOT_C$ is $\NP$-hard by Theorem~\ref{thm:main:exact}.
\end{proof}

\paragraph*{Acknowledgements.}
We are grateful to Pablo Parrilo, Philippe Rigollet, and Kunal Talwar for insightful conversations.

\small
\addcontentsline{toc}{section}{References}
\bibliographystyle{abbrv}
\bibliography{mot_hard}{}

\normalsize
\appendix
\section{Proof of Proposition~\ref{prop:hard-dft-approx}}

\begin{proof}[Proof of Proposition~\ref{prop:hard-dft-approx}] By Theorem~\ref{thm:main:approx}, it suffices to prove that approximating $\MinO_C(0)$ up to $\leq \Cmax/\poly(n,k)$ additive error is $\NP$-hard. By the reasoning in the proof of Proposition~\ref{prop:hard-dft}, it suffices to prove that \begin{equation}\min \left\{\frac{1}{2}\sum_{j \in S, j' \in S \sm \{j\}} U(\|x_j - x_{j'}\|_2, q_j, q_{j'}) : S \subset [n], |S| = k, \sum_{j \in S} q_j = 0\right\}
\label{eq:appadamson}
\end{equation} is $\NP$-hard to $\Cmax/\poly(n,k)$-approximate. We modify the proof of \citep[Theorem 5]{adamson2020hardness} to prove this. Note that for the parameters $A_{\pm}, B_{\pm}, C_{\pm}$ chosen in Lemma 1 of \citep[Theorem 5]{adamson2020hardness}, we have $\Cmax = \poly(n,k)$, and also the inequalities (1)-(3) of \citep{adamson2020hardness} are met with a small polynomial gap of at least $1/n^{10}$ in the sense that for large enough $n$,
\begin{align*}
	&\frac{A_+}{e^{B_+ n}} - \frac{C_+}{n^6} + \frac{1}{n} + \frac{A_-}{e^{B- \sqrt{1+n^2}}} - \frac{C_-}{(1+n^2)^3} - \frac{1}{\sqrt{1+n^2}} \geq \left|\frac{A_-}{e^{B_-}} - C_- - 1\right| + 1/n^{10} \\
	&n^2 \left|\frac{A_+}{e^{B_+ r}} - \frac{C_+}{r^6} + \frac{1}{r} + \frac{A_-}{e^{B- \sqrt{1+r^2}}} - \frac{C_-}{(1+r^2)^3} - \frac{1}{\sqrt{1+r^2}}\right| \leq \left|\frac{A_-}{e^{B_-}} - C_- - 1\right| - 1/n^{10}, & r \geq \sqrt{2n} \\
	& \frac{A_+}{e^{B_+ r}} - \frac{C_+}{r^6} + \frac{1}{r} + \frac{A_-}{e^{B- \sqrt{1+r^2}}} - \frac{C_-}{(1+r^2)^3} - \frac{1}{\sqrt{1+r^2}} > 1/n^{10}, & r \geq \sqrt{2n}.
\end{align*}
Tracing through the reasoning of Lemmas 2, 3, and 4 of \citep{adamson2020hardness}, this gap implies that a $\pm 0.49/n^{10}$ approximation to the objective~\eqref{eq:appadamson} suffices to determine whether or not the construction in \citep[Theorem 5]{adamson2020hardness} encodes a graph with an independent set of size $k/2$. This proves NP-hardness of 
$\Cmax/\poly(n,k)$
approximation for~\eqref{eq:appadamson}.
\end{proof}

\end{document}